\documentclass[12pt,reqno]{amsart}
\usepackage{color}
\usepackage{cleveref}
\usepackage{lipsum}
\usepackage{amssymb}
\usepackage{graphicx}
\usepackage{epstopdf}
\newtheorem{theorem}{Theorem}[section]

\newtheorem{proposition}[theorem]{Proposition}

\crefname{proposition}{Proposition}{lemmas}
\theoremstyle{definition}

\newtheorem{remark}[theorem]{Remark}

\numberwithin{equation}{section}

\begin{document}

\baselineskip=15.5pt

\title{ON COMPUTING LINEARIZING COORDINATES FROM \\SYMMETRY ALGEBRA}

\author[S. Ali]{Sajid Ali$^{1}$}

\address{$^{1}$School of Electrical Engineering and Computer Science, National University of Sciences and Technology, Islamabad 44000, Pakistan.}

\email{sajid{\_}ali@mail.com}

\author[H. Azad]{Hassan Azad$^{2}$}

\author[S.W. Shah]{Said Waqas Shah$^{2}$}

\address{$^{2}$Abdus Salam School of Mathematical Sciences, GCU, Lahore 54600, Pakistan.}

\email{hassan.azad@sms.edu.pk}


\email{waqas.shah@sms.edu.pk}

\author[F. M. Mahomed]{Fazal M. Mahomed$^{3,4}$}\footnote{
FM is Visiting Professor at UNSW for 2020}

\address{$^{3}$DSI-NRF Centre of Excellence in Mathematical and Statistical Sciences, School of Computer Science and Applied Mathematics, University of the Witwatersrand, Johannesburg, Wits 2050, South Africa}
\address{$^4$School of Mathematics and Statistics,
University of New South Wales,
Sydney NSW 2052 Australia}

\email{Fazal.Mahomed@wits.edu.za}


\date{}

\begin{abstract}
A characterization of the symmetry algebra of  the $N$th-order ordinary differential equations (ODEs) with maximal symmetry and all third-order linearizable ODEs is given.  This is used to show that such an algebra $\mathfrak{g}$ determines -- up to a point transformation -- only one linear equation whose symmetry algebra is $\mathfrak{g}$ and an algorithmic procedure is given to find the linearizing coordinates.  The procedure is illustrated by several examples from the literature.
\end{abstract}

\maketitle

\tableofcontents

\section{Introduction}\label{se1}
The Lie theory of transformation groups with applications to differential equations, both linear and nonlinear, was initiated by Sophus Lie \cite{lie1,lie2}. It is an elegant systematic approach to unravel continuous symmetries of differential equations. The determining equations for the symmetry group constitute a linear homogeneous overdetermined system of differential equations whose solutions do not rely on the solutions of the original equation if the equation is nonlinear.  Thus its utility is of paramount significance for nonlinear equations that admit symmetry. The symmetry/ies enable reductions and solutions as well as imply equivalence of the said equation to classified canonical forms. One can pursue the inverse problem which involves associating differential equations to realizations of vector fields of Lie algebras as well. Also they provide important practical and theoretical insights on the underlying differential equation. For linear equations, the symmetries are maximal and lead to the study of nonlinear equations that are linearizable via transformations.

This paper is a contribution to the problem of finding the linearizing coordinates for linearizable equations from the structure of the symmetry algebras. The main results are given as Theorems 3.4, 4.1 and 4.2 of Sections 3 and 4, respectively. These results use in an essential way the following results of Mahomed-Leach and Krause-Michel on linearizable ODEs.
\begin{theorem}
\cite{fazal, kra} (a) A necessary and sufficient condition for an $Nth$ order ODE ($N\geq 3$) to be linearizable by a point transformation is the existence of an $N-$dimensional abelian subalgebra of the symmetry algebra of the ODE.
\newline
(b) An $N$th-order ODE with $N\geq 3$ is linearizable by a point transformation if the dimension $m$ of its symmetry algebra is $N+4$, in which case the equation is equivalent to $y^{(N)}=0$. Moreover, if the equation is linearizable and $m\neq N+4,$ then $m$ must be $N+1$ or $N+2$.
\end{theorem}

The case $N=2$ was completely settled by Lie \cite{lie1}. He showed that a second-order ODE is linearizable if and only if its symmetry algebra is eight dimensional; in this case the symmetry algebra is $\mathfrak{sl}(3,\mathbb{R})$.  Lie \cite{lie1} also determined the symmetry algebra of $y^{(N)}=0$ for $N\ge3$.

Regarding third-order linearizable equations, such an equation can be transformed by a point transformation to one of the canonical forms given by $y'''=0$, $y^{\prime\prime\prime}-(\lambda+\mu )y^{\prime\prime}+\lambda \mu y^{\prime}=0$ $(\lambda,\mu \neq 0)$ or $y^{\prime\prime\prime}- ({\phi^{\prime\prime\prime}}/{\phi^{\prime\prime}})y''=0$, where $\phi^{\prime\prime} \neq 0$; see \cite{kra}.

The main contributions of this paper are: (1) an invariant  characterization of the symmetry algebra of the equation $y^{(N)}=0$ ($N>2$) and of all third-order linearizable ODEs; (2) algorithms for finding linearizing coordinates from a knowledge of the symmetry algebra of the equation.

We have used elementary representation theory to achieve this. The precise results are given in Sections 3 and 4 $-$ as \Cref{propnd}, \Cref{prop4d} and \Cref{prop5d}. Section 5 contains applications to linearizing coordinates for any third-order linearizable equations as well as examples for
higher-order equations with maximal symmetry $-$ using the algorithm given in \Cref{propnd} for $N$th-order equations with maximal symmetry.

A few words about the linearizability problem for ODEs are in order. Almost everything about symmetry analysis of ODEs is implicit in Lie \cite{lie2}.
However, it is desirable to have algorithms that  can be formulated more transparently and that can be implemented using programs like GAP, Maple and Sage.

The algebraic linearization and invariant criteria for scalar second-order ordinary differential equations (ODEs) are known from the classical works of Lie \cite{lie3} and Tress\'e \cite{tres}. These are given in terms of coefficients of the equation. Grissom et al \cite{grissom} used  the Cartan method to obtain these results. Moreover, a  geometric method of projection was utilized \cite{qad} to re-derive the Lie--Tress\'e linearization conditions. The reader is referred to, e.g., Ibragimov \cite{ibr3} and Olver \cite{olver} for modern expositions of symmetries and invariants.

For scalar $N$th-order ($N>2$), the Lie algebraic criteria for linearization by point transformations were deduced in \cite{fazal} and \cite{kra} as alluded above.  There are three canonical forms for such ODEs. In the case of third-order ODEs, the canonical forms are listed in \cite{Mahomed1996}.  The maximal algebra for such ODEs is dimension seven and this is for the ODE $y'''=0$ which is the simplest in its class. There are the cases of dimensions four and five as well.

Chern \cite{che} provided impetus for the Cartan equivalence method for linearization of certain classes of third-order ODEs by means of contact transformations.  In recent times, Neut and Petitot \cite{neu} obtained more general criteria for equivalence.

Grebot \cite{gre} considered the linearization of third-order ODEs by fibre preserving transformations which are restricted invertible transformations. Ibragimov and Meleshko \cite{ibr} recently re-looked at linearization for such ODEs.  They invoked a direct method. This works for point and contact transformations.

The Cartan equivalence method is currently in vogue \cite{Dweik3} and the authors arrived at an invariant characterization of scalar third-order ODEs for different dimensions of the symmetry Lie algebras.

The book of Schwarz \cite{sch} contains very useful information on algorithmic Lie theory as well as a large number of differential equations of higher orders of physical interest. The paper \cite{oud} of W. R. Oudshoorn and M. Van Der Put gives a new algorithm for computing the Lie symmetries for linear ODEs.

It is often difficult to find the linearizing coordinates with the approaches mentioned. The algorithmic approach proposed in this paper -- combined with the use of software -- is quite efficient, as can be seen from the examples given in the last section.  Notwithstanding, this method is algebraic in nature and is of significance.

\section{Review of real representations of $\mathfrak{sl}(2,\mathbb{R})$}
The algebra $\mathfrak{sl}(2,\mathbb{R})$ plays a special role in finite dimensional Lie algebras of local vector fields in the plane. The basic reason is that it is the only algebra that can occur as a proper Levi complement in such algebras: see Section 4 for an explanation of this fact. Thus the structure of the radical can be understood by using real representation theory of $\mathfrak{sl}(2,\mathbb{R})$.

The real finite dimensional representation theory of $\mathfrak{sl}(2,\mathbb{R})$  is identical to that of  complex finite dimensional representations of  $\mathfrak{sl}(2,\mathbb{C})$; indeed the real representation theory  for any real semisimple algebra with a split Cartan subalgebra  is identical to that of its complexification \cite{Ali}.

Let $X,Y,H$ be the standard basis of $\mathfrak{sl}(2,\mathbb{R})$ with $[X,Y]=H, [H,X]=2X, [H,Y]=-2Y$. If $V$ is a real representation space of $\mathfrak{sl}(2,\mathbb{R})$, then the dimension of the null space of $X$ in $V$ gives the number of irreducible components.
The element  $H$ operates as a real diagonalizable transformation in the null
space of $X$ with nonnegative integer eigenvalues. If $v_{1}, \ldots , v_{r}$ form a basis of eigenvectors of $H$ in the null space of $X$, with eigenvalues $d_{1},d_{2}, \ldots , d_{r}$ then the irreducible components are $\langle v_{1},Yv_{1}, \ldots , Y^{d_{1}}v_{1}\rangle, \ldots , \langle v_{r},Yv_{r}, \ldots , Y^{d_{r}}v_{r} \rangle$: the vectors $v_{1},\ldots , v_{r}$ are the high weight vectors of the representation: see Hilgert-Neeb \cite{hil}, and Knapp \cite{knapp} for further details.

\section{Characterization of the symmetry algebra of the equation $y^{(N)}=0$, $N\ge3$}

In this section we want to show that the structure of the symmetry algebra of $y^{(N)}=0$ ($N\geq 3)$, determines the equation and give an algorithmic procedure to find linearizing coordinates for any ODE of order $N$ whose symmetry algebra has the same Levi decomposition as that of the symmetry algebra of the equation $y^{(N)}=0$.  The Levi decomposition can be obtained by using standard programs in Computer Algebra Systems while the radical can be decomposed into its irreducible components using representation theory of $\mathfrak{sl}(2,\mathbb{R})$ -- as explained in Section 2.

From the classification of  Lie \cite{lie1} and Gonz\'alez-Lopez-Kamran-Olver \cite{gon}, it is clear that only $\mathfrak{sl}(2,\mathbb{R})$ can appear as a proper Levi complement. This can also be deduced from the following result. Before proving this result, let us note that by a proper Levi complement we mean a Levi complement which is neither the null subalgebra nor the whole algebra. Moreover the rank of a Lie algebra of vector fields is the dimension of its generic orbit; note that there is an open set on which all the orbits have a fixed dimension $r$ and the orbits in its complement have dimensions lower than $r$. We refer to this number $r$ as a rank of the Lie algebra of vector fields \cite{azad}.

\begin{proposition}
Let $L$  be a semisimple  Lie algebra of vector fields on an open set $U$ of $\mathbb{R}^K$. Assume that $L$ has a Cartan subalgebra $C$ with all real eigenvalues in the adjoint representation. Let $B\supset  C$ be a Borel subalgebra of $L$  and $\mathfrak{N}$ the nilpotent radical of  $B$.
If $\mathfrak{N}$  has an abelian subalgebra of rank $K$, then $L$ cannot occur as a proper Levi complement of any finite dimensional subalgebra of local vector fields on $\mathbb{R}^K$.
\end{proposition}
\begin{proof}  Because of the rank condition, one can introduce coordinates in which the abelian subalgebra of the nilpotent radical  is spanned by $\partial_{x_{i}},$ $i=1, \ldots , K$. Therefore since the radical is a sum of irreducibles for $L$, it must have a high weight vector in the radical -- if the radical is non-zero. This means that there is a vector field $v$ in the radical which commutes with $\partial_{x_{i}},$ $i=1, \ldots, K$ and therefore it must belong to the nilpotent radical of $L$. Therefore the radical must be 0.
\end{proof}
\begin{remark}
If we have a three-dimensional algebra of vector fields $X,Y $ and $Z$ with relations $[X,Y]=Z$ then we can always introduce coordinates so that one of the algebras $\langle X,Z \rangle$ or $\langle Y,Z \rangle$ is $\langle \partial_{x_{1}}, \partial_{x_{2}} \rangle$. The algebra $[X,Y]=Z$ is isomorphic to the nilpotent radical of a Borel subalgebra of $\mathfrak{sl}(3,\mathbb{R})$. Moreover using the main result of \cite{has} the nil radical of the algebra of vector fields in the plane isomorphic to  $\mathfrak{sl} (2, \mathbb{R})\times \mathfrak{sl} (2, \mathbb{R})$ is also of rank 2. Thus using Prop 3.1 we see why only $\mathfrak{sl}(2,\mathbb{R})$ can appear as a proper Levi complement in subalgebras of vector fields in the real plane.
\end{remark}
\begin{remark}
The condition of  a Cartan subalgebra being split can be relaxed if  the semisimple algebra consists of real analytic vector field -- by working in the complexification of real analytic vector fields on $U$.
\end{remark}
Recall that the Levi decomposition of the symmetry algebra of $y^{(N)}=0$ $(N\geq 3)$ is
\begin{equation}
\langle \partial_{x}, \frac{x^2}{2} \partial_{x} +\frac{N-1}{2}xy\partial_{y}, x\partial_{x}+\frac{N-1}{2}y\partial_{y} \rangle \oplus \langle \partial_{y} , x\partial_{y} , \ldots , x^{N-1} \partial_{y} , y\partial_{y} \rangle.
\end{equation}
 We note that the radical is a sum of  an  $N-$dimensional irreducible representation and a $1-$dimensional representation of the Levi complement with highest weight vectors $\partial_{y}$  and  $y\partial_{y}-$ relative to the Borel subalgebra $\langle \partial_{x}, x\partial_{x}+\frac{N-1}{2} y\partial_{y} \rangle$.  The radical is of rank 1 and the nilradical of the Borel subalgebra and the highest weight vector of the $N-$dimensional component give an abelian algebra of rank 2.

In the following proposition, we show that  these  conditions characterize the symmetry algebra of  $y^{(N)} =0$ $(N\geq 3)$.
\begin{theorem} \label{propnd}
Let $A$ be an algebra of  local vector fields in the plane with Levi decomposition $A=S\oplus R$, where $S$ has a basis $X,Y,H$ with $[X,Y]=H, [H,X]=2X, [H,Y]=-2Y$ and the radical $R$ is a sum of an irreducible $N$ dimensional representation  $(N\geq 2)$ and a 1-dimensional representation.

Let $V$  be a highest weight vector relative to the subalgebra $\langle H,X  \rangle$ for the $N$-dimensional component. Assume that $R$ is of rank 1 and the subalgebra $\langle X,V \rangle$ is of rank 2.  Then the canonical coordinates for  $\langle X,V \rangle$ -- after an affine change of  coordinates -- give coordinates in which the algebra coincides with the algebra of the equation  $y^{(N)}=0$, for $N\geq 3$ -- and this is the only hypersurface  of the form  $y^{(N)}=f(x,y,y', \ldots , y^{(N-1)})$ which is invariant  under $A$.
\end{theorem}
\begin{proof}
(a) Canonical form of the algebra: Choose coordinates $(x,y)$ so that $X=\partial_{x},V=\partial_{y}$. Let $Y=a\partial_{x}+b\partial_{y}$. Thus, $H=a_{x} \partial_{x}+b_{x} \partial_{y}$. Moreover, $[H,V]=(N-1)V$. By assumption on rank of the radical $[Y,V]=\lambda(x,y)V$. This gives $a_{y}=0$. Using $[H,\partial_{x}]=2\partial_{x}$ and $[H,\partial_{y}]=(N-1)\partial_y$, we see that $a''(x)=-2, b_{xx}=0, b_{xy}=-d,$ where $d=(N-1)$. Hence, $a(x)=-x^2 +\alpha x  +\beta$, $b(x,y)= -dxy+ \gamma x +H(y)$ and
$Y=(-x^2+\alpha x +\beta) \partial_{x}+ (-dxy+\gamma x + H(y) ) \partial_{y}$. Therefore, $H= [X,Y]= -2(x-\alpha /2)\partial_{x} -d (y- \gamma /d)\partial_{y}$. Put $\tilde{x}= x- \alpha /2 , \tilde{y} = y- \gamma /d$. In these coordinates, $H= -2\tilde{x} \partial_{\tilde{x}}-d \tilde{y} \partial_{\tilde{y}}$. Dropping the tildes for notational convenience, we have $X=\partial_{x}$, $Y= (-x^2 + \alpha^2/4 +\beta) \partial_{x} + (-dxy - d \alpha y/2+F(y))\partial_{y}$.

Finally, $[H,Y]=-2Y$, requires $\alpha^2/4+\beta=0$ and $-yF'(y) +F(y) = \alpha y - (2/d) F(y)$. The solutions of this equation are $F(y)= \frac{d\alpha}{2}y+ ky^{1+2/d}$. Therefore, $Y= -x^2 \partial_{x} + ( -dxy+k y^{1+2/d}) \partial_{y}$.

Lastly, we use the fact that the radical also has a 1-dimensional representation of the algebra $ \langle X,Y,H \rangle $. As the radical is of rank 1, the highest weight  vector of the trivial representation is necessarily $y\partial_{y}$. Now $[Y, y\partial_{y} ] = [ky^{1+2/d} \partial_{y}, y\partial_{y}]=0$ can only hold if $k=0$. Thus in the canonical coordinates of the algebra $\langle X,V \rangle $ and by a suitable translation, the algebra has the canonical form
\begin{equation}
\langle \partial_{x}, -({x^2} \partial_{x} +(N-1)xy\partial_{y} ), -(2x\partial_{x} +(N-1)y\partial_{y}) \rangle \oplus \langle \partial_{y} , x\partial_{y} , \ldots , x^{(N-1)} \partial_{y} , y\partial_{y} \rangle
\end{equation}

(b) Invariant hypersurfaces: To determine invariant hypersurfaces of the form $y^{(N)}=f(x,y,y', \dots , y^{(N-1)})$, we first determine invariant hypersurfaces under the maximal solvable subalgebra $B=\langle\partial_{x} , \partial_{y} ,x\partial_{y}, \dots , x^{(N-1)} \partial_{y} , y\partial_{y} \rangle$. We need formulas for the $N$-th order prolongations of these operators. They can be computed as in Section 5 of \cite{has}, following Lie \cite{lie1}.

Denoting $N$-th order prolongation of a vector field $X$ by $X^{(N)}$, we have ${\partial_{x}}^{(N)}= \partial_{x}$, $\partial_{y}^{(N)}= \partial_{y}$ and
$$(x^{a} \partial_{y})^{(N)} = x^{a} \partial_{y} + ax^{a-1}\partial_{y'}+ a(a-1) x^{a-2} \partial_{y''}+ \dots +\left (a(a-1)(a-2) \dots 2 \right) x \partial_{y^{(a-1)}}+ (a!) \partial_{y^{(a)}}$$ for $a\leq N$. Therefore, $(x^{a} \partial_{y})^{(N)}$ is congruent to $(a!) \partial_{y^{(a)}} ~ \bmod ~ \partial_{y}, \ldots ,\partial_{y^{(a-1)}}$. Invariance of the hypersurface $y^{(N)}=f(x,y,y', \ldots , y^{(N-1)})$ under $\partial_{x}^{(N)}$, $\partial_{{y}}^{(N)}$ and $(x^{a} \partial_{y})^{(N)}$ (for $a\leq N$) shows that the equation must be $y^{(N)} =K$. Applying the $N$-th order prolongation of $y\partial_{y}$ $-$ namely $y\partial_{y}+y'\partial_{y'}+ \ldots +y^{(N)}\partial_{y^{(N)}}$, we see that $y^{(N)}=0$ on the hypersurface  $y^{(N)}=K$.
Thus $K$ must be 0.

When $N\geq 3$, the symmetry algebra of $y^{(N)}=0$  is the algebra $A$, while for $N=2$, the symmetry algebra of the equation  contains $A$
as a subalgebra. The algebra $A$ is a parabolic subalgebra in the sense that it contains a maximal solvable subalgebra of the symmetry algebra of $y''=0$.
\end{proof}

The above proposition gives an algorithm for linearizing an ODE with maximal symmetry. This will be implemented in Section 5 for several equations from the literature.

\section{Linearizing coordinates of third-order ODEs}
We will consider in this section the cases where the symmetry algebra is solvable, because in the previous section an algorithm to find linearizing coordinates for any equation equivalent to $y^{(N)}=0$ ($N\geq 3$) has already been given. The canonical forms for equations with solvable symmetry algebras are: $y^{\prime\prime\prime}- ({\phi^{\prime\prime\prime}}/{\phi^{\prime\prime}})y''=0$, where $\phi^{\prime\prime} \neq 0$ \, and \,
$y^{\prime\prime\prime}-(\lambda+\mu )y^{\prime\prime}+\lambda \mu y^{\prime}=0$ $(\lambda, \mu \neq 0)$. The corresponding symmetry algebras are $\mathfrak{g}_{1}= \langle \partial_{y}, x\partial_{y},y\partial_{y},\phi(x) \partial_{y} \rangle $ and $\mathfrak{g}_{2}= \langle \partial_{x}, \partial_{y},y\partial_{y},e^{\lambda x} \partial_{y},e^{\mu x} \partial_{y}>$ $-$ respectively: this is for the case when $\lambda, \mu$ are real and distinct, the remaining cases are given in Remark 4.3 below.

The algebra $\mathfrak{g}_{1}$ is of rank 1, its commutator is abelian and $\mathfrak{g}_{1}/\mathfrak{g}_{1}'$ operates on $\mathfrak{g}_{1}$ with eigenvalue $-1$ of multiplicity 3 while $\mathfrak{g}_{2}'$ has an abelian complement, say $A$, of rank 2, and the weights of $A$ on $\mathfrak{g}_{2}'$ $-$ for an ordered basis of $A$ $-$
are $(0,-1)$, $(\lambda,-1)$ , $(\mu, -1)$. These conditions characterize the algebras and the differential equations -- intrinsically.
\begin{theorem} \label{prop4d}
Let $\mathfrak{g}$ be a four-dimensional algebra of vector fields in the plane which is of rank 1 whose commutator is abelian and such that $\mathfrak{g}/\mathfrak{g}^{\prime}$ operates on $\mathfrak{g}'$ with a non-zero
eigenvalue of multiplicity 3. Then, one can algorithmically introduce coordinates $(x,y)$ so that $\mathfrak{g}'=\langle\partial_{y}, x\partial_{y}, \phi(x)\partial_{y} \rangle$ and $\mathfrak{g} = \langle y \partial_{y}, \partial_{y}, x\partial_{y}, \phi(x) \partial_{y} \rangle $. The generic orbit of $\mathfrak{g}$ in the extended space
$(x,y,y^{\prime}, y^{\prime\prime}, y^{\prime \prime \prime})$ is four-dimensional and such points are in the complement of the set defined by
$y^{\prime \prime} \phi^{\prime \prime \prime} - \phi^{\prime \prime} y^{\prime \prime \prime} =0. $ In these coordinates, this is the only hypersurface invariant under $\mathfrak{g}$.
\end{theorem}
\begin{proof}
By assumption $\mathfrak{g}/\mathfrak{g}'$ operates on $\mathfrak{g}'$ with a nonzero eigenvalue of multiplicity 3. By scaling, we can assume that this eigenvalue is $-1$. Let the eigenvectors be $f_{1},f_{2},f_{3}$. The whole algebra is of rank 1 and the commutator is abelian. Therefore, we can introduce coordinates $(x,y)$ so that
$f_{1}=\partial_{y}, f_{2}= x\partial_{y},f_{3}=\phi(x) \partial_{y}$, with $1,x,\phi(x)$ linearly independent. Let $q$ be a representative of  $\mathfrak{g}/\mathfrak{g}^{\prime}$, say $q=f(x,y)\partial_{y}$. Then, $[q, \partial_{y}]=-\partial_{y}$ implies $\frac{\partial f}{\partial y}=1$. Hence,
$f=y+g(x)$. Let $\tilde{x}=x, \tilde{y}= y+g(x)$. In these coordinates, $\partial_{\tilde{x}}= \partial_{x} -g'(x), \partial_{\tilde{y}}= \partial_{y}$. Thus $\mathfrak{g}$ has the form stated in the statement of the proposition. For notational convenience, we write $\mathfrak{g}$ in its canonical form as $\mathfrak{g}=\langle y\partial_{y}, \partial_{y}, x\partial_{y},\phi(x) \partial_{y} \rangle $, where $\phi''\neq 0$.

Now the generic orbits of a Lie algebra of  analytic vector fields $L$ on an open set of $\mathbb{R}^{N}$ are the same as that of the algebra of vector fields obtained from the matrix of coefficients of $L$ by putting it in reduced row echelon form \cite{merker} (p.102, Proposition 4), \cite{azad}. Thus the rank of the coefficient matrix gives the dimension of the generic orbit.

We want to describe such orbits in the extended space $(x,y,y',y'',y''')$.

The third prolongation of $\mathfrak{g}$ can be computed as in Section 5 of \cite{Ali}, following Lie \cite{lie1} (pp. $261-274$).  As the third-order prolongation of $f(x)\partial_{y}$  is  $f(x)\partial_{y}+f'(x) \partial_{y'}+ f''(x) \partial_{y''}+ f'''(x) \partial_{y'''}$ and the third prolongation of $y\partial_{y}$ is $y\partial_{y}+y'\partial_{y'}+y''\partial_{y''}+y'''\partial_{y'''}$, the rank of the vector fields coming from the given basis of $\mathfrak{g}$  in the third prolongation is the rank of
\begin{equation}
\begin{bmatrix}
1 & 0 & 0 & 0 \\
x & 1 & 0 & 0 \\
y & y' & y'' & y''' \\
\phi & \phi' & \phi'' & \phi'''
\end{bmatrix}
\end{equation}
Thus the generic orbit is four-dimensional on the complement of the hypersurface defined by $y''\phi'''-y'''\phi''=0$. Therefore, bringing the algebra into its canonical form determines the equation $y''' = (\phi'''/\phi'') y''$.

We now show that this is the only equation in these coordinate of the form $y'''=f(x,y,y',y'')$ that is invariant under the third prolongation of  $\mathfrak{g}$.
Invariance under  $\partial_{y}$ and $x\partial_{y}+ \partial_{y'}$ gives $y'''=F(x,y'')$.  Invariance under the third order prolongation of $y\partial_{y}$  shows that $y'''=y''F_{y''}$ on the hypersurface $y'''=F(x,y'')$.  Thus  $y''F_{y''} =F(x,y'')$. Working with canonical coordinates of the field $y''\partial_{y''}$ we find that  $F(x,y'') = y'' G(x)$.

Finally, invariance under the third-order prolongation of $\phi(x)\partial_{y}$ shows that $\phi'''= \phi''G(x)$. Thus the invariant equation must be  $y'''=y''({\phi'''}/{\phi''})$.
\end{proof}
\textbf{Note.} This shows that expressing the symmetry algebra $L$ of a third-order ODE whose algebra is described as in the statement of \Cref{prop4d} in the coordinates $(x,y)$ described in the proof of this proposition linearizes the equation.
\begin{theorem} \label{prop5d}
{Let $\mathfrak{g}$ be a five-dimensional Lie algebra of vector fields in the plane whose commutator $\mathfrak{g}^{\prime}$ is abelian and
of rank 1. Assume that $\mathfrak{g}^{\prime}$ has an abelian complement $A$. Assume also that the weights of $\mathfrak{g}/\mathfrak{g}^{\prime}$ on $\mathfrak{g}^{\prime}$ with respect to an ordered basis $(\bar{e}_{1},\bar{e}_{2})$, $e_{i} \in A$ of $\mathfrak{g}/\mathfrak{g}^{\prime}$, are
$(0, k)$, $(\lambda,k)$, $(\mu, k)$ with $k,\lambda,\mu\neq 0$ and $\lambda \neq \mu$ with weight vectors $e_{3},e_{4},e_{5}$. Moreover, assume that if $w$ is the weight vector of weight $(0,k)$, then $\langle Z_{A}(w),w \rangle$ is of rank 2. Then in the canonical coordinates for $\langle Z_{A}(w),w \rangle$ the algebra $
\mathfrak{g} =  \langle \partial_{x}, x\partial_{y} \rangle \oplus \langle \partial_{y}, e^{\lambda x}\partial_{y}, e^{\mu x}\partial_{y} \rangle.
$
The generic $\mathfrak{g}$ -- orbit in the extended space $(x,y,y^{\prime}, y^{\prime \prime}, y^{\prime \prime \prime})$ is five-dimensional and points with lower
dimension orbits are those which satisfy the equation $
y^{\prime\prime\prime}-(\lambda+\mu)y^{\prime\prime}+\lambda\mu y^{\prime}=0$.  Moreover, in these coordinates this is the only hypersurface invariant under $\mathfrak{g}$.
}
\end{theorem}
\begin{proof}
As the  proof is very similar to the proof of the previous proposition, we will omit some details.  Using the notations in the statement of the proposition,  let $e_{3},e_{4},e_{5}$  be vectors in $\mathfrak{g}'$ of weights $(0,k),(\lambda, k) , (\mu,k)$ for $A$.  By assumption, the centralizer of  $e_{3}$ in $A$ - namely $\langle e_{3},e_{1} \rangle$ is of rank 2. Choose coordinates so that $e_{1}=\partial_{x},e_{3} = \partial_{y}$.  By commutativity of $\mathfrak{g}'$ and the assumption that  $\mathfrak{g}'$ is of rank 1, we have $e_{4}=f(x)\partial_{y},e_{5} = g(x)\partial_{y}$.

As $e_{4}, e_{5}$  are eigenvectors for $\partial_{x}$ with eigenvalues $\lambda, \mu$  we may assume that $e_{4} = e^{\lambda x} \partial_{y}, e_{5} = e^{\mu x} \partial_{y}$. Now  $e_{2}$ operates on $\mathfrak{g}'$ as a nonzero scalar $k$. By scaling $e_2$ , we may assume that $k=-1$.  Moreover as $e_2$ commutes with $e_1$ , it  must be of the form $e_{2} = a(y) \partial_x +b(y) \partial_{y}$. Thus $[e_2, e_3]=-e_3$, gives $a(y) = \alpha , b(y) = y + \beta$.
Finally, $[e_2, e_4]=-e_4$ gives $\alpha =0$ and  $e_2 = y+\beta $.  Thus in the coordinates $\tilde{x}=x,\tilde{y}= y+\beta$, the algebra $\mathfrak{g}$  is of the form given in the statement of the proposition.

Dropping the tildas for notational convenience, we see after computing the third prolongations of the basis of  $\mathfrak{g}$, that the coefficient matrix of the generators of $\mathfrak{g}$ is singular at the points in the extended space $(x,y,y',y'',y''')$ where $y'\lambda \mu -y'' (\lambda +\mu)+y'''=0$. Arguing as in the previous proposition,  the generic orbit outside this linear subspace is five-dimensional.

Suppose that $y'''= F(x,y,y',y'')$ is an invariant hypersurface for $\mathfrak{g}$. Invariance under the prolongations of $\partial_{y}$  and $\partial_{x}$ shows  $F$ is a function of only $y',y''$.  Using invariance under the third prolongation of $y\partial_{y}$ we see that $y'''=y' F_{y'}+y''F_{y''}$ on the set $y'''=F(y',y'')$ . Thus  $y'F_{y'}+y''F_{y''}=F$.

Working with canonical coordinates for the field $y'\partial_{y'}+y''\partial_{y''}$ we see that $F(y',y'')= y''G(y'/y'')$. Invariance under the third order prolongation of $e^{ax}\partial_{y}$  implies that $G(w)$ is a linear function of $w$. Thus the equation of the invariant hypersurface is $y'''=y''(\alpha x+ \beta y'/y'')$. Hence invariance under $\langle e^{\lambda x}\partial_{y}, e^{\mu x}\partial_{y} \rangle$  determines the values of $\alpha, \beta$  and the equation of the invariant hypersurface is $y'''-(\lambda+\mu) y''+ \lambda \mu y'=0$.
\end{proof}
\begin{remark}
The canonical forms of the cases when $\lambda$ is complex and non-real or when $\lambda =\mu$ are $\mathfrak{g}_{1}= \langle \partial_{x}, y\partial_{y} \rangle \oplus \langle \partial_{y} ,e^{ax} \cos{(bx)}\partial_{y},e^{ax} \sin{(bx)}\partial_{y} \rangle, $ where $\lambda =a+\sqrt{-1}\,b$ and $\mathfrak{g}_{2}= \langle \partial_{x}, y\partial_{y}\rangle \oplus \langle\partial_{y} ,e^{\lambda x} \partial_{y},x e^{\lambda x} \partial_{y} \rangle$, respectively. The proofs of these assertions are similar to the proof given above and are therefore omitted.
\end{remark}

\section{Applications}
We now consider various nonlinear differential equations of order $N\geq 3$ and determine their linearizing coordinates using  \Cref{propnd}, \Cref{prop4d}, \Cref{prop5d}. With the help of advanced packages Differential Geometry and Lie Algebras in Maple 18, we algorithmically implement these propositions followed by the determination of symmetry algebras. An advantage of this approach is that it relies only on the structure of symmetry algebras and a knowledge of the structure constants suffices to determine linearizing coordinates. Apart from the examples given below, reduction to canonical forms of linearizable equations in Schwarz \cite{sch} -- e.g. p. 234 -- p. 245 -- and equations on p. 83 of Kamke (see \cite{kamke,kamke1}) can also be linearized as applications of the results of this paper.
\newline\newline
\textbf{a. Third-Order ODEs}
There are three classes of linearizable third-order ODEs corresponding to symmetry algebras of dimensions 4, 5 and 7, of which the first two symmetry algebras are solvable while the other algebra is non-solvable. Below we consider examples from each of these classes and linearize them accordingly.
\newline
\emph{Class 1:}
Consider the nonlinear ODE
\begin{equation}
y''' - \frac{(x+3 -3(x+2))y'' - ((x+2)y^{'}-x-3)y'^{2}}{x+2}=0,\label{ode3}
\end{equation}
which has a solvable four-dimensional symmetry algebra of rank 1
\begin{equation}
e_{1} = \partial_{y}, ~e_{2} = e^{-y}\partial_{y}, ~ e_{3} = xe^{-y}\partial_{y}, ~e_{4} = e^{x-y}\partial_{y}.
\end{equation}
where we invoke \Cref{prop4d} to linearize it. Its derived algebra is $\langle e_{2},e_{3},e_{4} \rangle$ of rank 1 and $\mathfrak{g}/\mathfrak{g}'$ whose representative is $\partial_{y}$, operates on $\mathfrak{g}'$ with eigenvalues $-1$ of multiplicity 3. We choose coordinates so that $e_{2}=\partial_{\tilde{y}}$ and $e_{3}= \tilde{x}\partial_{\tilde{y}}$. Therefore,  $\tilde{x}=x,\tilde{y}(x)=e^{y(x)}$, and ODE (\ref{ode3}) is reduced to its canonical form given in Gat \cite{gat}
\begin{equation}
\tilde{y}''' -\frac{\tilde{x}+3}{\tilde{x}+2} \tilde{y}^{\prime\prime} =0.
\end{equation}
\emph{Class 2a:}
A nonlinear ODE
\begin{equation}
y''' - \frac{3y^{\prime\prime 2}}{y^{'}}-x y'^{4}=0, \label{ode2}
\end{equation}
has a five-dimensional solvable Lie algebra
\begin{equation}
e_{1} = \partial_{y}, ~e_{2} = x\partial_{x}, ~ e_{3} = e^{-y}\partial_{x}, ~e_{4} = e^{y/2}\sin ( \sqrt{3} \,y/2 )\partial_{x}, ~e_{5} = e^{y/2}\cos ( \sqrt{3} \,y/2 )\partial_{x},
\end{equation}
with relations
\begin{align}
&[e_1, e_3] = -e_3, [e_1, e_4] = (\sqrt3 /2) e_5+(1/2)e_4, [e_1, e_5] = (1/2)e_5-(\sqrt{3} /2) e_4, \nonumber \\
&[e_2, e_3] = -e_3, [e_2, e_4] = -e_4, [e_2, e_5] = -e_5\,.
\end{align}
Its derived algebra is
\begin{equation}
\mathfrak{g}^{'} = \langle e_{3}, \frac{1}{2} ( e_{4} + \sqrt{3} e_{5}) ,  \frac{1}{2} ( e_{5} - \sqrt{3} e_{4}) \rangle.
\end{equation}
We compute the centralizer of $e_{3}$ in $\langle e_{1},e_{2} \rangle$ which is $\langle e_1 - e_{2} \rangle$. Moreover, $e_3$ and $e_{1}-e_{2}$ is of rank 2.
The element $e_2$ has the eigenvalues $-1$ on $\mathfrak{g}'$ and $e_{1}-e_2$ has eigenvalues $0$ and $\pm \sqrt{3}$. Thus, by using the \Cref{prop5d}, the canonical coordinates corresponding to algebra $\langle e_{3}, e_{1}-e_{2} \rangle$, are $x=e^{-r}v(r), ~y(x) = r$, that transforms nonlinear ODE (\ref{ode2}) into
\begin{equation}
v^{\prime \prime \prime} -3v^{\prime \prime} +3v^{\prime}=0.
\end{equation}

\emph{Class 2b:}
The nonlinear ODE
\begin{equation}
y'y'''=xy'^{5} +y'^{4} -y''y'^2 +3 y''^2, \label{ode2b}
\end{equation}
has a five-dimensional solvable Lie algebra
\begin{equation}
e_{1} = \partial_{y}, ~e_{2} = x\partial_{x}, ~ e_{3} = \sin{y} \partial_{x}, ~e_{4} = \cos{y}\partial_{x}, ~e_{5} = e^{-y} \partial_{x}\,,
\end{equation}
with derived algebra $\mathfrak{g}' = \langle e_{3},  e_{4} ,  e_{5}  \rangle$. The one- and two-dimensional invariant subspaces of $A=\langle e_1, e_2 \rangle $ are $\langle e_5 \rangle $ and $\langle e_3, e_4 \rangle $, respectively. By introducing the new basis $\tilde{e}_{1}=e_1 - e_2, \tilde{e}_{2}=e_2$, we find that $[\tilde{e}_{1}, e_3]= e_4 +e_3, [\tilde{e}_{1}, e_4]= e_4 - e_3, [\tilde{e}_{1},e_5]=0$, and $e_2$ operates on $\mathfrak{g}'$ as a multiplication by $-1$. Moreover, the centralizer of $e_{5}$ in $A$ is $\tilde{e}_{1}$. The eigenvalues of $\tilde{e}_{1}$ on the space spanned $\langle e_{3},e_{4} \rangle$ are $1\pm i$. By using \Cref{prop5d} and Remark 4.3, the canonical coordinates corresponding to algebra $\langle \tilde{e}_{1} , e_{5} \rangle $, will give us the linearizing coordinates. The canonical coordinates are given by $x=e^{-r}v(r), ~y(x) = r$, which transforms nonlinear ODE (\ref{ode2b}) into
\begin{equation}
v^{\prime \prime \prime} -2v^{\prime \prime} +2v^{\prime}=0.
\end{equation}

\emph{Class 3:}
We now apply our algebraic tools to linearize a third-order ODE (Example on p.82, Kamke \cite{kamke}) which also involves an arbitrary parameter
\begin{equation}
y^2 y^{\prime\prime \prime } + ayy'y''+by'^3=0, \ \label{ode4}
\end{equation}
where $~9b=a^2-3a, ~a\neq -3$. It has a seven-dimensional Lie algebra generated by the vector fields
\begin{align}
&e_{1} = x\partial_{x}, ~e_{2} = \partial_{x}, ~ e_{3} = y\partial_{y}, ~e_{4} = y^{-a/3}\partial_{y}, e_{5} = x{y}^{-a/3}\partial_{y},  \nonumber \\
&e_{6} = x^2{y}^{-a/3}\partial_{y},~e_{7} = \frac{x^2}{2}\partial_{x}+\frac{3xy}{a+3}\partial_{x}\,.
\end{align}
The Levi decomposition is
\begin{equation}
\mathfrak{g}= r \oplus s = \langle e_3, e_4, e_5, e_6 \rangle \oplus \langle e_1+ \frac{3}{a+3}e_3, e_2, e_7 \rangle.
\end{equation}
We now bring the semisimple part to the standard form of $\mathfrak{sl}(2,\mathbb{R}),$ by calling $H=-2e_{1}-6/(a+3) e_{3}$, $X=-2e_{2}$ and $Y=e_{7}$. We have $[H,e_{3}]=0$ and $[H,e_{6}]=-2e_{6}$. Moreover, $\langle e_{6}, e_{7} \rangle $ forms a rank 2 algebra. This gives us the canonical coordinates
\begin{equation}
\tilde{x}=\frac{6y^{(a+3)/3}}{(a+3) x^2},~ \tilde{y}(\tilde{x})=-\frac{2}{x},
\end{equation}
that convert the equation into $\tilde{y}'''=0$.

\textbf{c. Fourth-Order ODE} We now linearize a fourth-order ODE given in \cite{ibr1} using \Cref{propnd}. The equation
\begin{equation}
y^{3} y^{(4)} +4k y^2 y^{\prime} y^{\prime\prime\prime}+3ky^2 y^{\prime\prime 2} +6k(k-1) y y^{\prime 2} y^{\prime \prime} +k(k-1)(k-2)y^{\prime 4}=0. \label{4order1}
\end{equation}
has symmetry algebra
\begin{align}
&e_{1} = x\partial_{x}, ~ e_{2}	 = \partial_{x}, ~e_{3} = y\partial_{y}, ~ e_{4} = y^{1-a} \partial_{y}, ~ e_{5} = xy^{1-a} \partial_{y},\nonumber\\
& e_{6} = x^2 y^{1-a} \partial_{y},~ e_{7} = x^3 y^{1-a} \partial_{y},~ e_{8} =ax^2\partial_{x}+ 3xy \partial_{y},
\end{align}
where $a=k+1$. The Levi decomposition is
\begin{equation}
\mathfrak{g} = \langle e_{3},e_{4},e_{5},e_{6}, e_{7} \rangle  \oplus \langle e_{1}+(3/2a) e_{3}, e_{2}, e_{8} \rangle.
\end{equation}
In order to bring the semisimple part to the standard form of $\mathfrak{sl}(2,\mathbb{R})$, we introduce $H=-2 e_{1}-3e_{3}/a$, $X=-2e_{2}$ and $Y=e_{8}$. Now $ \langle e_{2},e_{4} \rangle $ form a rank 2 algebra which gives us the linearizing coordinates $\tilde{x}=x, ~ \tilde{y}(\tilde{x})  = y^{a}/a $ that linearize the ODE \eqref{4order1} and bring it in the required form $\tilde{y}^{\prime \prime \prime \prime} =0$.
 \newline\newline
\textbf{d. Fifth-Order ODE} An example of a fifth-order ODE that can be linearized using the same procedure is given below. Consider the maximal algebra ODE
\begin{equation}
y^{\prime 3} y^{(5)} - 15 y^{\prime 2} y^{\prime \prime} y^{(4)} -10 y^{\prime 2} y^{ \prime \prime \prime 2} + 105y^{\prime} y^{\prime \prime 2} y^{\prime
\prime \prime} -105 y^{\prime \prime 4}=0 \label{5order},
\end{equation}
which has Lie algebra generators
\begin{align}
&e_{1} = \partial_{x}, ~ e_{2} = x\partial_{x}, ~e_{3} = y\partial_{y}, ~ e_{4} = y \partial_{x}, ~ e_{5} = y^2 \partial_{x}, \nonumber\\
& e_{6} = y^3 \partial_{x},~ e_{7} = y^4 \partial_{x},~ e_{8} =4xy\partial_{x}+ y^2 \partial_{y}.
\end{align}
As before we use \Cref{propnd}, to find the linearizing coordinates $\tilde{x}=y(x), \tilde{y}(\tilde{x})=x,$ that transforms ODE \eqref{5order} into $\tilde{y}^{(5)}=0$.

{\bf Acknowledgments}. FMM thanks the N.R.F. of RSA for funding support.



\begin{thebibliography}{ZZZZZZZ}
\bibitem{lie1} S. Lie, Vorlesungen uber Differentialgleichungen mit bekannten infinitesimalen Transformationen, BG Teubner, 1891.

\bibitem{lie2} S. Lie, Theorie der Transformationsgruppen, Vol. 3, https://eudml.org/doc/202686.

\bibitem{fazal} F. M. Mahomed and P. G. L. Leach, Symmetry Lie algebras of $nth$ order ordinary differential equations, J. Math. Anal. Appl., Vol. 151, 90-107, 1990.

\bibitem{kra} J. Krause and L. Michel, Lecture Notes Phys. 382, 251, 1991.

\bibitem{lie3} S. Lie, Klassifikation und Integration von gew\"onlichen Differentialgleichungenzwischen $x$, $y$, die eine Gruppe von Transformationen gestaten,  Arch. Math.,  VIII, IX, 187, 1883.

\bibitem{tres} A. Tress\'e, Sur les Invariants Diff\'erentiels des Groupes Continus de Transformations,  Acta Math.,  18, 1, 1894.

\bibitem{grissom} C. Grissom, G. Thompson and G. Wilkens, Linearization of Second Order Ordinary Differential Equations via Cartan's Equivalence Method, Journal of Differential Equations, Chapman $\&$ Hall CRS, Boca Raton, 2008.

\bibitem{qad} F. M. Mahomed and A. Qadir, Invariant linearization criteria for systems of cubically nonlinear second-order ordinary differential equations, Journal of Nonlinear Mathematical Physics, 16, 283, 2009.

\bibitem{ibr3} N. H. Ibragimov, A Practical Course in Differential Equations and Mathematical Modelling, Classical and
New Methods, Nonlinear Mathematical Models, Symmetry and Invariance Principles, Higher Education Press, World Scientific, New Jersey, 2009.

\bibitem{olver} P. J. Olver, Equivalence, Invariance and Symmetry, Cambridge University Press, 1995.

\bibitem{Mahomed1996} N. H. Ibragimov and F. M. Mahomed, Ordinary differential equations, CRC Handbook of Lie Group Analysis of Differential Equations, vol. 3. N. H. Ibragimov ed., CRC Press, Boca Raton, 191, 1996.

\bibitem{che} S. S. Chern, The geometry of the differential equation $y'''=F(x,y,y,y'')$, Sci. Rep. Nat. Tsing Hua Univ. 4,  97-111, 1940.

\bibitem{neu} S. Neut and M. Petitot, La g\'eom\'etrie de l'\'equation $y'''=f(x,y,y',y'')$, C.R. Acad. Sci. Paris S\'er I, {335}, 515-518, 2002.

\bibitem{gre} G. Grebot, The characterization of third order ordinary differential equations admitting a transitive fibre-preserving point symmetry group, {J. Math. Anal. Applic.}, {206},  364-388, 1997.

\bibitem{ibr} N. H. Ibragimov and S. V. Meleshko, Linearization of third-order ordinary differential equations by point and contact transformations, {J. Math. Anal. Appl.}, {308}, 266-289, 2005.

\bibitem{ibr1} N. H. Ibragimov, S. V. Meleshko and S. Suksern, Linearization of fourth-order ordinary
differential equations by point transformations, J. Phys. A: Math. Theor., 41, 235206-19, 2008.

\bibitem{Dweik3} Ahmad Y. Al-Dweik, M. T. Mustafa, Fazal Mahmood Mahomed and Rajai S. Alassar, Linearization of third order differential equations $u'''=f(x,u,u',u'')$ via point transformations, Mathematical Methods in the Applied Sciences, Vol 41 (16), 6017-7098, 2018.

\bibitem{sch} F. Schwarz, Algorithmic Lie Theory for Solving ordinary Differential Equations, Chapman $\&$ Hall  CRS, Boca Raton, 2008.


\bibitem{gat} O. Gat, Symmetry algebras of third-order ordinary differential equations, Journal of Mathematical Physics, 33, 2966, 1992.

\bibitem{oud} W. R. Oudshoorn and M. Van Der Put, Lie Symmetries and Differential Galois Groups of Linear Equations, Math. Comp., 71 , 349, 2001.

\bibitem{Ali} S. Ali, H. Azad and I. Biswas, A constructive method for decomposing real representations, https://arxiv.org/abs/1804.03614.

\bibitem{hil} J. Hilgert and K. H. Neeb, Structure and geometry of Lie groups, Springer Monographs in Mathematics, Springer, New York, 2012.

\bibitem{knapp} A. W. Knapp, Lie groups beyond an introduction, Second edition, Progress in Mathematics, 140. Birkhauser Boston, Inc., Boston, MA, 2002.

\bibitem{gon} A. G. Lopez, N. Kamran and P. J. Olver, Lie Algebras of vector fields in the real plane, {Proc. London Math. Soc.}, (3) 64, 339-368, 1992.

\bibitem{has} H. Azad, I. Biswas and F. M. Mahomed, Equality of the algebraic and geometric ranks of Cartan subalgebras and applications to linearization of a system of ordinary differential equations, Int. J. Math. 28, 1750080, 2017.

\bibitem{merker} J. Merker, Theory of Transformation Groups by Sophus Lie and Friedrich Engel (Vol. I, 1888), https://arxiv.org/pdf/1003.3202.pdf.

\bibitem{azad} H. Azad, I. Biswas, R. Ghanam and M. T. Mustafa, On computing joint invariants of vector fields, J. of Geom. Phys., 97, 69-76, 2015.

\bibitem{kamke} E. Kamke, Differentialgleichungen: Losungsmethoden und Losungen, I, Gewohnliche Differentialgleichungen, B. G. Teubner, Leipzig, 1977.

\bibitem{kamke1} J. D. Ke{\v c}ki{\'c}, Additions to Kamke's treatise: Nonlinear third order differential equations, http://pefmath2.etf.rs/files/93/396.pdf.

\end{thebibliography}
\end{document}